\documentclass[a4paper,10pt]{amsart}
\usepackage{amsmath,amssymb,amsthm}
\usepackage[alphabetic, abbrev]{amsrefs}
\usepackage{enumerate}
\usepackage[all]{xy}
\newdir{ >}{{}*!/-5pt/@{>}} 
\SelectTips{cm}{} 
\usepackage{fullpage}
\usepackage{hyperref}

\numberwithin{equation}{section}

\newtheorem{theorem}[subsection]{Theorem}
\newtheorem{corollary}[subsection]{Corollary}
\newtheorem{lemma}[subsection]{Lemma}
\newtheorem{proposition}[subsection]{Proposition}
\theoremstyle{definition}
\newtheorem{definition}[subsection]{Definition}

\newtheorem{remark}[subsection]{Remark}

\newtheorem{example}[subsection]{Example}
\newtheorem{examples}[subsection]{Examples}

\newcommand{\bC}{\mathbb{C}}

\newcommand{\Z}{\mathbb{Z}}
\newcommand{\F}{\mathbb{F}}
\newcommand{\Q}{\mathbb{Q}}

\newcommand{\cL}{\mathcal{L}}

\newcommand{\HH}{\mathsf{HH}}
\newcommand{\HR}{\mathsf{HR}}
\newcommand{\THH}{\mathsf{THH}}
\newcommand{\ssets}{\mathsf{sSets}}

\DeclareMathOperator{\Tor}{Tor}

\newcommand{\ie}{\emph{i.e.}}

\newcommand{\ra}{\rightarrow}

\title{Stability of Loday constructions}
\author{Ayelet Lindenstrauss}
\address{Mathematics Department, Indiana University, 831 East Third Street,
  Bloomington, IN 47405, USA}  
\email{alindens@indiana.edu}
\author{Birgit Richter} 
\address{Fachbereich Mathematik der Universit\"at Hamburg,
  Bundesstrasse 55, 20146 Hamburg, Germany} 
\email{birgit.richter@uni-hamburg.de}

\date{\today}

\begin{document}

\begin{abstract}
We study the question for which commutative ring spectra $A$ the
tensor of a simplicial set $X$ with $A$, $X \otimes A$, is  a
stable invariant in the sense that it depends only on the homotopy
type of $\Sigma X$. We prove several
structural properties about different notions of stability,
corresponding to different levels of invariance required of $X\otimes
A$. 
We establish stability in important cases, such as complex and
real periodic topological K-theory, $KU$ and $KO$. 
\end{abstract}
  
\maketitle

\section{Introduction}
For any simplicial set $X$ and any commutative ring spectrum $A$ one
can form the tensor of $X$ with $A$, 
$X \otimes A$. An important special case of this construction is
topological Hochschild homology of $A$, $\THH(A)$, which is $S^1
\otimes A$. In the following we will often work with commutative
$R$-algebras for some commutative ring spectrum $R$.  We will sometimes take
coefficients in a commutative
$A$-algebra $C$, which requires working with \emph{pointed} simplicial
sets $X$; we denote the corresponding object (whose definition we recall in Section
\ref{defofcl} below)
by $\cL_X^R(A; C)$.  When $C=A$,
 $\cL_X^S(A;A)$ is just $X \otimes A$, and $\cL_{S^1}^R(A; C)$
is $\THH^R(A; C)$. 

As topological Hochschild homology is the target
of a trace map from algebraic K-theory 
\begin{equation} \label{eq:trace}
  K(A) \ra \THH(A)
\end{equation}
it has been calculated in many cases. Higher order topological Hochschild homology,
which is $\cL_{S^n}^R(A; C)$, has also been determined 
in many important classes of examples, see for instance
\cite{bhlprz,dlr,hhlrz,schlichtkrull,stonek}. In \cite{bhlprz} we
develop several tools for calculating $\cL_{\Sigma
  X}^R(A;C)$. However, if we want to determine the homotopy type of
$\cL_X^R(A; C)$ and $X$ doesn't happen to be a suspension, then the
range of methods is 
much sparser.

Rognes' redshift conjecture \cite{redshift} predicts that applying algebraic
K-theory raises chromatic level by one in good cases. In particular,
higher chromatic phenomena could be detected by iterated algebraic
K-theory of rings.  
If $A$ is a commutative ring spectrum, then so are $K(A)$ and
$\THH(A)$, and as the trace map is 
a map of commutative ring spectra, one can iterate
the trace map from \eqref{eq:trace} to obtain
\[ K(K(A)) \ra \THH(\THH(A))\]
and one doesn't have to stop at two-fold iterations. As $X \otimes A$
is the tensor of $A$ with $X$ in the category of commutative ring
spectra \cite[chapter VII, \S 2, \S 3]{ekmm}, one can identify
\[ \THH(\THH(A)) = S^1 \otimes (S^1 \otimes A)\]
with $(S^1 \times S^1) \otimes A$ and this is \emph{torus homology of
  $A$}. Similarly, any $n$-fold iteration of algebraic K-theory of $A$
has an iterated trace map to $(S^1)^n \otimes A$. There are
calculations of torus homology of $H\F_p$ for small $n$ by Rognes,
Veen \cite{veen} and Ausoni-Dundas, but a general result is
missing. However, the homotopy type of $S^n \otimes H\F_p$ is known for
every $n$ and for small $n$  $(S^1)^n \otimes A$ splits as follows: We have that 
$\Sigma (S^1)^n \simeq \Sigma (\bigvee_{i=1}^n \bigvee_{\binom{n}{i}} S^1)$ and
one obtains for small $n$
\[ (S^1)^n \otimes H\F_p \simeq (\bigvee_{i=1}^n \bigvee_{\binom{n}{i}} S^1) \otimes H\F_p.\]

This gave rise to the question whether $\cL_X^R(A; C)$ is a
\emph{stable invariant}, \ie,  whether the 
homotopy type of $\cL_X^R(A; C)$ only depends on the homotopy type of
$\Sigma X$. There are positive results: $\cL_X^{Hk}(HA)$ is a stable
invariant if $k$ is a field and $A$ is a commutative Hopf algebra over
$k$ \cite[Theorem 1.3]{bry} or if $k$ is an arbitrary commutative ring
and $A$ is a smooth $k$-algebra \cite[Example 2.6]{dt}. But Dundas and Tenti also 
show \cite[\S 3.8]{dt} that $\cL_X^{H\Q}(H\Q[t]/t^2)$ is \emph{not} a
stable invariant. They show that  
$\cL_{S^1 \vee S^1 \vee S^2}^{H\Q}(H\Q[t]/t^2)$ and $\cL_{S^1 \times
  S^1}^{H\Q}(H\Q[t]/t^2)$ differ and that reducing the coefficients
from   $H\Q[t]/t^2$  to  $H\Q$ doesn't eliminate
this discrepancy. This also implies that $\cL_X^S(H\Q[t]/t^2)$ and
$\cL_X^S(H\Q[t]/t^2; H\Q)$ are not stable invariants because $S_\Q \simeq H\Q$.  

Our aim is to investigate the question of stability in a systematic
manner. We start by defining  several different notions of stability. Instead of asking for
equivalent homotopy types  
of $\cL_X^R(A;C)$ and $\cL_Y^R(A;C)$ if $\Sigma X \simeq \Sigma Y$ we
are asking when we actually  
get an equivalence $\cL_X^R(A;C) \simeq \cL_Y^R(A;C)$ of augmented
commutative $C$-algebras. There  
are intermediate notions that ask for less structure to be preserved,
for instance, that the  
equivalence $\cL_X^R(A;C) \simeq \cL_Y^R(A;C)$ is one of commutative $R$-algebras
or of $C$- or $R$-modules.  

We establish that stability is preserved by several constructions such
as base-change and products but we also  
show which procedures do \emph{not} preserve stability. For instance
stability is not a transitive property: if $R \ra A$ and $A \ra B$
satisfy stability then this does not imply that $R \ra B$ has this property.

A central purpose of this paper is to establish new cases where
stability holds. For instance  
for any regular quotient $R \ra R/(a_1,\ldots, a_n)$ of a commutative
ring $R$ we obtain  
stability for the induced map of commutative ring spectra $HR \ra
HR/(a_1, \ldots, a_n)$. 
Free commutative ring spectra generated by a module spectrum satisfy
stability and we suggest a notion of \emph{really smooth} maps of
commutative ring spectra. These are maps $R \ra A$ that can be factored
as the canonical inclusion of $R$ into a free commutative $R$-algebra
spectrum followed by  a map that satisfies \'etale descent, so these
maps model the local behaviour of smooth maps in the context of
algebra, compare \cite[Proposition E.2 (d)]{loday}. We show 
that really smooth maps satisfy stability. Other examples where
stability holds are Thom spectra as well as $S \ra KU$ and other spectra of the
form $S \ra R_h = (\Sigma^\infty_+ W_h)[x^{-1}]$ considered in \cite{chy}. Using Galois descent
we also obtain stability for $S \ra KO$. 

For calculations like that of torus homology, one often doesn't really need stability, but the
property of the suspension to decompose products is the crucial
feature that one wants to have on the level of $\cL^R_{(-)}(A;
C)$. Therefore we say that $R \ra A \ra C$ decomposes products if 
\[ \cL^R_{X \times Y}(A; C) \simeq \cL_{X \vee Y \vee X \wedge Y}^R(A;
  C)\]
for all pointed simplicial sets $X$ and $Y$. 
We use Greenlees' spectral sequence \cite[Lemma 3.1]{greenlees} in the
case $C = Hk$ for $k$ a field to show that 
this decomposition property is preserved under forming suitable retracts. 

In Section \ref{sec:rational} we close with some observations on stability in characteristic zero, using
that rationally the suspension of  pointed simply connected simplicial sets splits into a pointed
sum of rational spheres and using \cite[Proposition 4.2]{bry} where Berest, Ramadoss and Yeung describe
the behaviour of representation homology and higher order Hochschild homology under rational
equivalences. 

\subsection*{Acknowledgement} 
We thank Bj{\o}rn Dundas for many helpful discussion and for spotting several dumb mistakes in earlier
versions of 
this paper. The second named author thanks the Isaac Newton Institute for Mathematical Sciences for support and 
hospitality during the programme \emph{K-theory, algebraic cycles and motivic homotopy theory} when work on
this paper was undertaken. This work  was supported by Simons Collaboration Grant 359565 for the first author and EPSRC grant number EP/R014604/1 for the second.

\subsection{Definition of $\cL^R_X(A; C)$} \label{defofcl}

We denote the category of simplicial sets by $\ssets$ and the one of pointed simplicial sets
by $\ssets_*$. 
Let $X$ be a finite pointed simplicial set and let $R \ra A \ra C$ be a sequence
of maps of commutative ring spectra. We assume that $R$ is a cofibrant
commutative $S$-algebra and that $A$ and $C$ are cofibrant commutative
$R$-algebras. The cofibrancy assumptions on $R$, $A$ and $C$ will
ensure that the homotopy type of $\cL_X^R(A; C)$ is well-defined: 

The \emph{Loday construction with respect to $X$ of
$A$ over $R$ with coefficients in $C$} is the simplicial commutative
augmented $C$-algebra spectrum $\cL^R_X(A; C)$ whose $p$-simplices are 
\[ C \wedge \bigwedge_{x \in X_p \setminus *} A \]
where the smash products are taken over $R$.  
Here, $*$ denotes the basepoint of $X$ and we place a copy of $C$ at
the basepoint. As the smash product over $R$ is the coproduct in the
category of commutative $R$-algebra spectra, the simplicial structure
is straightforward: Face maps $d_i$ on $X$ induce multiplication in
$A$ or the $A$-action on $C$ if the basepoint is 
involved. The degeneracy maps $s_i$ on $X$ cause the insertion of
the unit map $\eta_A \colon R \ra A$ over all $n$-simplices which
are not hit by $s_i\colon X_{n-1} \to X_n$. As defined,
$\cL^R_X(A; C)$ is a simplicial commutative augmented 
$C$-algebra spectrum.  We use the same symbol $\cL^R_X(A;
C)$ for its geometric realization. For $C=A$ we abbreviate 
$\cL^R_X(A; A)$ by $\cL^R_X(A)$. 

For $X = S^n$ we write $\THH^{[n], R}(A; C)$ for $\cL^R_{S^n}(A; C)$  and
if $R=S$, then we omit it from the notation, so $\THH^{[n]}(A; C) = \cL^S_{S^n}(A; C)$. For $n=1$
this is the classical case of topological Hochschild homology of $A$ with coefficients in $C$,
$\THH(A; C)$. 

Note that $\cL^R_X(A)$ is by definition \cite[VII, \S 2, \S3]{ekmm} equal to
$X \otimes A$ where $X \otimes A$ is formed in the category of commutative
$R$-algebras.

If $X \in \ssets_*$ is an arbitrary object, then we can write it as 
the colimit of its finite pointed subcomplexes and the Loday
construction with respect to $X$ can then also be expressed as the
colimit of the Loday construction for the finite pointed subcomplexes.

\section{Notions of stability}
The weakest notion of stability just asks for an abstract equivalence
in the stable homotopy category: 
\begin{definition}
  \begin{enumerate}
  \item[]
    \item
  Let $R \ra A$ be a cofibration of commutative $S$-algebras with $R$
  cofibrant. We call $R \ra A$ 
  \emph{stable} if for every pair of pointed simplicial sets $X$ and $Y$
  an equivalence $\Sigma X \simeq \Sigma Y$ implies that $\cL_X^R(A)
  \simeq \cL_Y^R(A)$. 
\item
  Let $S \ra R \ra A \ra B$ be a sequence of cofibrations of
  commutative $S$-algebras. Then we call 
  $(R, A, B)$ \emph{stable}, if for every pair of pointed simplicial
  sets $X$ and $Y$ 
  an equivalence $\Sigma X \simeq \Sigma Y$ in $\ssets_*$ implies that $\cL_X^R(A;
  B) \simeq \cL_Y^R(A; B)$. 
  \end{enumerate}
\end{definition}

\begin{examples}
\begin{itemize}
\item[]
\item  
  Dundas and Tenti show that for any discrete smooth $k$-algebra $A$
  we have that $Hk \ra HA$ is stable \cite[Example 2.6]{dt}.
\item
  They show, however, that $H\Q \ra H\Q[t]/t^2$ and $(H\Q, H\Q[t]/t^2, H\Q)$
  are not stable.
\item
  If $A$ is a commutative Hopf algebra over a field $k$, then Berest,
  Ramadoss and Yeung prove \cite[\S 5]{bry} that 
  $Hk \ra HA$ and $(Hk, HA, Hk)$ are stable by comparing higher order Hochschild homology to
  representation homology. For a purely homotopy-theoretic proof see \cite[Theorem 3.8]{hklrz}.
\item
  In \cite{bhlprz} we show that for any sequence of cofibrations of
  commutative $S$-algebras $S \ra A \ra B \ra A$ we get that
  \[ \cL_X^B(A) \simeq \cL_{\Sigma X}^A(B; A) \]
  as augmented commutative $A$-algebras and hence $B \ra A$ is stable
  if $B$ is a cofibrant commutative augmented $A$-algebra.   
 \end{itemize} 
\end{examples}
  
In the above definition we just require an abstract weak equivalence,
but one can also pose additional conditions on the equivalence
  $\cL_X^R(A; B) \simeq \cL_Y^R(A; B)$. A strong version of
  stability is the following:  
\begin{definition}
  \begin{enumerate}
  \item[]
    \item
  Let $R \ra A$ be a cofibration of commutative $S$-algebras with $R$
  cofibrant. We call $R \ra A$ 
  \emph{multiplicatively stable} if for every pair of pointed
  simplicial sets $X$ and $Y$ 
  an equivalence $\Sigma X \simeq \Sigma Y$ in $\ssets_*$  implies that $\cL_X^R(A)
  \simeq \cL_Y^R(A)$ as commutative augmented $A$-algebra spectra.
\item
  Let $\xymatrix@1{S \ar[r] & R \ar[r]^\alpha & A \ar[r]^\beta & B}$
  be a sequence of cofibrations of commutative $S$-algebras. Then we call
  $R \ra A \ra B$ \emph{multiplicatively stable} if for every pair of
  pointed simplicial sets $X$ and $Y$ an equivalence $\Sigma X \simeq
  \Sigma Y$ in $\ssets$ implies that $\cL_X^R(A; B) \simeq \cL_Y^R(A; B)$  and
  $\cL^R_X(B) \simeq \cL^R_Y(B)$ as commutative augmented $B$-algebras
  such that the diagram
  \[ \xymatrix{
\cL_X^R(A; B) \ar@{-}[rr]^\simeq \ar[d]_{\cL^R_X(\beta)}& &
\cL_Y^R(A; B)\ar[d]^{\cL^R_X(\beta)} \\ 
\cL^R_X(B) \ar@{-}[rr]^\simeq & & \cL^R_Y(B)    }\]
commutes.   
\end{enumerate}
\end{definition}
Of course, there is a whole hierarchy of notions of
stability. Instead of asking that the equivalence $\cL_X^R(A)
  \simeq \cL_Y^R(A)$ is one of augmented commutative $A$-algebras, we
  could ask for one of augmented commutative $R$-algebras or $A$- or
  just $R$-modules.
\begin{definition}
  Let $R \ra A$ be a cofibration of commutative $S$-algebras with $R$
  cofibrant. We call $R \ra A$ 
  \emph{$A$-linearly stable} if for every pair of pointed
  simplicial sets $X$ and $Y$ 
  an equivalence $\Sigma X \simeq \Sigma Y$ in $\ssets_*$ implies that $\cL_X^R(A)
  \simeq \cL_Y^R(A)$ as $A$-modules. Similarly, we call $R \ra A$ 
  \emph{$R$-linearly stable} if for every pair of pointed
  simplicial sets $X$ and $Y$ 
  an equivalence $\Sigma X \simeq \Sigma Y$ in $\ssets_*$ gives rise to an
  equivalence of $R$-modules $\cL_X^R(A) \simeq \cL_Y^R(A)$.

\end{definition}
\begin{remark}
If $R \ra A$ is $A$-linearly stable, then  $(R, A, B)$
is stable because 
\[ \cL_X^R(A; B) \simeq \cL_X^R(A) \wedge_A B. \]
If $R \ra A$ is multiplicatively stable, then so is $R \ra A \ra B$ for every cofibrant commutative
$A$-algebra $B$. 

A converse might not be true: Even if $B$ is faithful as an
$A$-module, we might not know that the equivalence $\cL_X^R(A)
\wedge_A B \simeq \cL_Y^R(A) \wedge_A B$ is of the form $f \wedge_A
B$, so we cannot deduce that $\cL_X^R(A) \simeq \cL^R_Y(A)$. 
\end{remark}

Let us start with several examples of multiplicative stability.  
\begin{proposition} \label{prop:augmented}
  If $B$ is an augmented commutative $A$-algebra, then $B \ra A$ and
  $A \ra \cL_{\Sigma X}^A(B; A) \ra A$ are multiplicatively stable. 
\end{proposition}
\begin{proof}  
In the augmented case $A \ra B \ra A$, as an equivalence $\Sigma X \simeq \Sigma Y$ in $\ssets_*$ implies
that $\cL_{\Sigma X}^A(B;A) \simeq \cL_{\Sigma Y}^A(B;A)$ as augmented commutative $A$-algebras, we
also
get that $\cL_X^B(A) \simeq \cL_Y^B(A)$ as augmented commutative $A$-algebras by applying
\cite[Theorem 3.3]{bhlprz} to the sequence of maps $A = A \ra B \ra A$, so $B \ra A$ is
multiplicatively stable. 

For the second claim observe that 
\[ \cL_Y^A(\cL_{\Sigma X}^A(B; A); A) \simeq \cL_{Y \wedge \Sigma X}^A(B; A) =
  \cL_{\Sigma Y \wedge X}^A(B; A). 
\]
Observe that for all $X$ we have that $\cL^A_X(A) \simeq A$
so $A \ra \cL_{\Sigma X}^A(B; A) \ra A$ is multiplicatively stable. 
\end{proof}

Loday constructions for suspensions are stable:   
\begin{theorem} \label{thm:lxstable}
  Let $R \ra A$ be a cofibration of commutative $S$-algebras with $R$
  cofibrant. Then $A \ra \cL_{\Sigma X}^R(A)$ is multiplicatively
  stable for all $X$.  
\end{theorem}
\begin{proof}
  We have to show that $\cL_Y^A(\cL^R_{\Sigma X}(A))$ only depends on
  the homotopy type of $\Sigma Y$. We first
  identify $\cL_Y^A(\cL^R_{\Sigma X}(A))$ with the help of
  \cite[Remark 3.3]{hhlrz} as  
  \begin{align*}
    \cL_Y^A(\cL^R_{\Sigma X}(A)) \simeq & \cL^R_Y(\cL^R_{\Sigma X}(A))
                                          \wedge_{\cL^R_Y(A)} A \\ 
    \simeq & \cL^R_{Y \times \Sigma X}(A) \wedge_{\cL^R_Y(A)} A \\
    \simeq & \cL^R_{(Y \times \Sigma X) \cup_{Y} *}(A) \\
    \simeq & \cL^R_{Y_+ \wedge \Sigma X}(A) \cong \cL^R_{\Sigma(Y_+) \wedge X}(A).
  \end{align*}
As $\Sigma(Y_+) \simeq \Sigma Y \vee S^1$ for $Y \in \ssets_*$, this depends only on $\Sigma Y$. 
    \end{proof}
\begin{example} \label{ex:thh2fp}  
  Applying Theorem \ref{thm:lxstable} to $H\F_p$ and $\Sigma X = S^2$
  gives that the map
  \[ H\F_p \ra \THH^{[2]}(H\F_p) \simeq H\F_p \vee \Sigma^3 H\F_p\]
  is multiplicatively stable for all primes $p$. 
\end{example}

As we know from the algebraic setting that smooth algebras are stable, it is natural to consider
free commutative $A$-algebra spectra. 
Let $M$ be an $A$-module spectrum for some commutative $S$-algebra $A$. We consider the free
commutative $A$-algebra on $M$,
\[ \mathbb{P}_A(M) = \bigvee_{n \geq 0} M^{\wedge_A n}/\Sigma_n\]
with the usual convention that $M^{\wedge_A 0}/\Sigma_0 = A$.

In the following we use several categories, so let's fix some notation. Let $\mathcal{U}$ denote
the category of unbased (compactly generated weak Hausdorff) spaces. For a commutative ring
spectrum $R$, $\mathcal{M}_R$ denotes the category of $R$-module spectra and $\mathcal{C}_R$
denotes the category of commutative $R$-algebras.

\begin{lemma} 
  For every simplicial set $X$ there is a weak equivalence of commutative $A$-algebras 
  \[ \cL^A_X(\mathbb{P}_A(M)) \simeq \mathbb{P}_A(X_+ \wedge M).\]  
\end{lemma}
\begin{proof}
  For the proof we use the fact that the category of commutative $A$-algebras is tensored over
  unpointed
  topological spaces and simplicial sets in a compatible way \cite[VII \S 2, \S 3]{ekmm}. Note that
  $\cL^A_X(\mathbb{P}_A(M)) = X \otimes_A \mathbb{P}_A(M)$ in the notation of \cite{ekmm}.

  We have the following chain of bijections for an arbitrary commutative $A$-algebra $B$: 
  \begin{align*}
    \mathcal{C}_A(X \otimes_A \mathbb{P}_A(M), B) & \cong
 \mathcal{U}(X, \mathcal{C}_A(\mathbb{P}_A(M), B)) \\ 
  & \cong \mathcal{U}(X, \mathcal{M}_A(M, B)) \\ 
  & \cong \mathcal{M}_A(X_+ \wedge M, B) \\ 
    & \cong \mathcal{C}_A(\mathbb{P}_A(X_+ \wedge M), B)
  \end{align*}
  where $X_+ \wedge M$ is the tensor of $X$ with $M$ in the category
  of $A$-modules. Hence the Yoneda lemma implies the claim. 
\end{proof}

\begin{corollary} \label{cor:pstable}
In the setting above, if $\Sigma X \simeq \Sigma Y$, then
$\cL_X^A(\mathbb{P}_A(M)) \simeq \cL_Y^A(\mathbb{P}_A(M))$ as 
commutative $A$-algebras. 
\end{corollary}
\begin{proof} 
If $\Sigma X \simeq \Sigma Y$, then $\Sigma^\infty_+ X \simeq
\Sigma^\infty_+ Y$ and as $X_+ \wedge M =
  \Sigma^\infty_+ X \wedge M$ this implies that  $\mathbb{P}_A(X_+
  \wedge M) \simeq \mathbb{P}_A(Y_+ \wedge M)$ as 
  commutative $A$-algebras. 
\end{proof}

The following example was also considered in \cite[Lemma 5.5]{mccm}. A
cofibration $A \ra B$ of commutative $S$-algebras with $A$ cofibrant
is called $\THH$-\'etale if the canonical map $B \ra
\THH^A(B)$ is a weak equivalence.

\begin{proposition}
If $A \ra B$ is $\THH$-\'etale, then for all connected pointed $X$ the
canonical map $B \ra \cL_X^A(B)$ is an equivalence. Hence, as this map is a
map of augmented commutative $B$-algebras, $\cL_X^A(B) \simeq
\cL_Y^A(B)$ for any pair of connected simplicial sets $X$ and $Y$. 
\end{proposition}
\begin{proof}
  The proof is by induction on the top dimension of a non-degenerate simplex in a finite connected
  simplicial set, and then by taking colimits in the infinite case.  A connected
  $0$-dimensional simplicial set consists of a point, where there is nothing to prove.
  Any $1$-dimensional connected finite simplicial set is homotopy equivalent to a wedge of
  circles, so if $X\simeq S^1\vee S^1\vee\ldots\vee S^1$ and $B\simeq \cL_{S^1}^A B$,  
\[\cL_{X}^A B\simeq B\wedge_B B \wedge_B\cdots \wedge_B B\simeq B. \]
Once we know the result for simplicial sets of dimension $\leq n-1$, if we get a simplicial set
$X$ with a finite number of non-degenerate $n$-cells we proceed by induction on the number of
those.   As in the proof of Proposition 8-4 in \cite{bhlprz}, using the homotopy invariance of
the construction and subdivision, if needed, we can assume that $X$ can be constructed by adding
a new non-degenerate simplex with an embedded boundary to a simplicial set homotopy equivalent to
$X$ with one non-degenerate $n$-cell deleted, for which the proposition holds by the induction on
the number of non-degenerate $n$-cells.  By the inductive hypothesis it also holds for the
embedded boundary $\partial\Delta^n$, and since the new simplex being added is homotopy equivalent
to a point, the proposition holds for it.  By the connectivity and by homotopy invariance we can
also assume that the basepoint of $X$ is contained in  the boundary of the new simplex being
attached, so the identifications of all three Loday constructions with $B$ are compatible. Then
$\cL_X^A(B)\simeq B\wedge_B B\simeq B$.
\end{proof}
\begin{remark}
  Examples of $\THH$-\'etale maps $A \ra B$ are Galois extensions in
  the sense of \cite{rognes} but also \'etale maps in the sense of
  Lurie \cite[Definition 7.5.1.4]{lurie}. For a careful discussion of
  these notions and for comparison results see \cite{mathew}.
\end{remark}

\section{Inheritance properties and descent}

With the assumption of multiplicative stability we get a descent result: 
\begin{theorem}
  If $R \ra A \ra B$ is multiplicatively stable, then $A \ra B$
  is multiplicatively stable. 
\end{theorem}
\begin{proof}
  Let's assume that $\Sigma X \simeq \Sigma Y$ in $\ssets_*$. Then by assumption we
 get that 
  $\cL_X^R(B) \simeq \cL_Y^R(B)$ and $\cL^R_X(A;B) \simeq \cL^R_Y(A;
  B)$ as commutative augmented $B$-algebras, compatibly with the module structure of the former
  over the latter. The Juggling Lemma
  \cite[Lemma 3.1]{bhlprz} yields an equivalence of augmented
  commutative $B$-algebras 
  \[  \cL_X^A(B) \simeq B \wedge_{\cL_X^R(A;B)} \cL^R_X(B) \text{ and }
      \cL_Y^A(B) \simeq B \wedge_{\cL_Y^R(A;B)} \cL^R_Y(B). \]
    Our assumptions guarantee that therefore $\cL_X^A(B) \simeq
    \cL_Y^A(B)$ as commutative augmented $B$-algebras. 
\end{proof}
One can upgrade this slightly and introduce coefficients:
\begin{corollary}
  If $S\ra R\ra A\ra B\ra C$ is a sequence of cofibrations of commutative $A$-algebras and both $R \ra A \ra C$ and $R \ra B \ra C$ are multiplicatively stable,
  then $A \ra B \ra C$ is multiplicatively stable as well. 

\end{corollary}

\begin{lemma} \label{lem:basechange}
Let $A \leftarrow R \ra B$ be cofibrations of commutative $S$-algebras
with $R$ cofibrant, then  
$\cL_X^A(A \wedge_R B) \cong A \wedge_R \cL_X^R(B)$ as simplicial commutative augmented
$A \wedge_R B$-algebras and
hence on realizations as commutative augmented $A \wedge_R B$-algebras.
\end{lemma}
\begin{proof}
There is a direct isomorphism sending $A \wedge_R (B \wedge_R \ldots
\wedge_R B)$ to $(A \wedge_R B) \wedge_A \ldots \wedge_A (A \wedge_R
B) $  
and this isomorphism is compatible with the multiplication. 
\end{proof} 
This implies that stability is closed under base-change: 
\begin{proposition} \label{prop:induction}
  Let  $A$ and $E$ be cofibrant commutative $R$-algebra spectra. If $R
  \ra A$ is $R$-linearly stable, then so is $E \ra E \wedge_R A$. If
  $R \ra A$ is multiplicatively stable, then so is $E \ra E
  \wedge_R A$. 
\end{proposition}
\begin{proof}
  Assume that $\Sigma X \simeq \Sigma Y$ in $\ssets_*$. Then by assumption
  $\cL^R_X(A) \simeq \cL^R_Y(A)$ as $R$-modules or as augmented
  commutative $A$-algebras. But then also $E \wedge_R \cL^R_X(A)
  \simeq E \wedge_R \cL^R_Y(A)$ and by Lemma \ref{lem:basechange} this implies
\[\cL_X^{E}(E \wedge_R A) \simeq \cL_Y^{E}(E \wedge_R A).\]
\end{proof}

\begin{remark}
Note that the above implication cannot be upgraded to an equivalence:
starting with the assumption that  $\cL_X^{E}(E \wedge_R A) \simeq
\cL_Y^{E}(E \wedge_R A)$, we get $E \wedge_R \cL^R_X(A) \simeq E
\wedge_R \cL^R_Y(A)$. Even if $\cL^R_X(A)$ and $\cL^R_Y(A)$ are $E$-local
in the category of $R$-modules, however, we don't know that the weak
equivalence 
$E \wedge_R \cL^R_X(A) \simeq E \wedge_R \cL^R_Y(A)$ is of the form 
$E \wedge_R f$ (or a zigzag of such maps), but for the $E$-local
Whitehead Theorem \cite[Lemma 1.2]{bousfield} we have to have a map and not
just an abstract isomorphism of $E_*$-homology groups. 
\end{remark}
Smashing with a fixed commutative $R$-algebra preserves stability: 
\begin{lemma}
  Let $A$, $B$ and $C$ be cofibrant commutative $R$-algebras. Then
  there is an equivalence of commutative augmented $C \wedge_R B$-algebras
  \[\cL_X^{C \wedge_R A}(C \wedge_R B) \simeq C \wedge_R \cL_X^A(B). \]
  Hence if $f \colon A \ra B$ is multiplicatively stable, then so is $C \wedge_R f \colon
  C \wedge_R A \ra C \wedge_R B$. 
\end{lemma}
\begin{proof}
  The equivalence \[\cL_X^{C \wedge_R A}(C \wedge_R B) \simeq C
    \wedge_R \cL_X^A(B)\] is based on the equivalence
  \[ (C \wedge_R B) \wedge_{(C \wedge_R A)} (C \wedge_R B) \simeq C
    \wedge_R (B \wedge_A B).\] 
\end{proof}
\begin{proposition} \label{prop:regular}
  Let $R$ be a commutative ring and let $a \in R$ be a regular
  element. Then $HR \ra HR/a$ is multiplicatively stable. 
\end{proposition}
\begin{proof}
  We consider the pushout $HR \wedge^L_{HR[t]} HR$ where the right
  algebra map $R[t] \ra R$ sends $t$ to zero
  and the left algebra map sends $t$ to $a$. Note that with respect to
  both of these maps $HR[t]$ is an augmented
  commutative $HR$-algebra spectrum. The K\"unneth spectral sequence for
  $\pi_*(HR \wedge^L_{HR[t]} HR)$ has as its $E^2$-term
  $\mathrm{Tor}^{R[t]}_{*,*}(R, R)$ and we take the standard free
  $R[t]$ resolution 
  \[ \xymatrix@1{ 0 \ar[r] & R[t] \ar[r]^{t} & R[t]} \]
  of $R$. Applying $(-)\otimes_{R[t]} R$ yields
  \[ \xymatrix{ 0 \ar[r]& R[t] \otimes_{R[t]} R \ar[d]_\cong 
\ar[rrr]^{t \otimes \mathrm{id}}_{=\mathrm{id}
        \otimes a}
      && &  R[t] \otimes_{R[t]} R  \ar[d]^\cong \\
      0 \ar[r] & R \ar[rrr]^a &&& R.       } \]
  Note, that the regularity of $a$ is needed to ensure injectivity on the left hand side.

  We apply Lemma \ref{lem:basechange} and choose a cofibrant model of $HR$ as a commutative
  $HR[t]$-algebra
  and obtain
  \[ \cL^{HR}_X(HR/a) \simeq \cL^{HR}_X(HR \wedge_{HR[t]} HR) \simeq
    HR \wedge_{HR[t]} \cL^{HR[t]}_X(HR)\] 
  where the right $HR[t]$-module structure of $\cL^{HR[t]}_X(HR)$
  factors through the augmentation map 
  sending $t$ to $0$. Assume that $\Sigma X \simeq \Sigma Y$ in $\ssets_*$.
  By Proposition \ref{prop:augmented} we have that $HR[t] \ra HR$ is
  multiplicatively stable, so $\cL^{HR[t]}_X(HR) \simeq
  \cL^{HR[t]}_Y(HR)$ as commutative augmented  
  $HR$-algebras. This yields an equivalence of commutative augmented
  $HR \wedge_{HR[t]} HR \simeq HR/a$-algebras between
  $\cL^{HR}_X(HR/a)$ and $\cL^{HR}_Y(HR/a)$.  
\end{proof}  
\begin{remark}
The above result can be used for calculating torus homology for instance for $H\Z \ra
H\Z/p\Z$ for every prime $p$: We know the homotopy type of $\cL_{S^k}^{H\Z}(H\Z/p\Z)$ by
\cite[Proposition 5.3]{bhlprz} for all $k$ and therefore we get the homotopy type of
$\cL_{(S^1)^n}^{H\Z}(H\Z/p\Z)$ as smash products over $H\Z/p\Z$ of $\binom{n}{k}$ copies of
$\cL^{H\Z}_{S^k}(H\Z/p\Z)$ for $1 \leq k \leq n$.
\end{remark}

\begin{corollary} \label{cor:squarezero}
  For every commutative ring $R$ and every regular element $a \in R$
  the square-zero extension  
  \[ HR/a \ra HR/a \vee \Sigma HR/a\] is multiplicatively stable. In
  particular, for every commutative ring $R$, $HR \ra HR \vee  \Sigma
  HR$ is multiplicatively stable. 
\end{corollary}
\begin{proof}
  As $HR \ra HR/a$ is multiplicatively stable we get by Lemma
  \ref{lem:basechange} that 
  \[ \cL_X^{HR/a}(HR/a \wedge_{HR} HR/a) \simeq HR/a \wedge_{HR}
    \cL_X^{HR}(HR/a)\] 
 as augmented commutative $HR/a \wedge_{HR} HR/a$-algebras and hence
 $HR/a \ra HR/a \wedge_{HR} HR/a$ is multiplicatively stable. The
 K\"unneth spectral sequence yields 
  that $\pi_*(HR/a \wedge_{HR} HR/a) \cong \Lambda_{R/a}(x)$ with $|x|
  =1$. By \cite[Proposition 2.1]{dlr} this implies that
  \[ HR/a \wedge_{HR} HR/a \simeq HR/a \vee \Sigma HR/a\]
  as a commutative augmented $HR/a$-algebra. 

Considering the regular element $t \in R[t]$ gives that $HR \ra HR
\vee \Sigma HR$ is multiplicatively stable. 

\end{proof}
Stability is inherited by Loday constructions.
\begin{proposition} \label{prop:lodaystable}
If $R \ra A$ is multiplicatively stable, then so is $R \ra \cL_Z^R(A)$
for any $Z$.  
\end{proposition}
\begin{proof}
  Assume that $\Sigma X \simeq \Sigma Y$. As $\Sigma (X \times Z)
  \simeq \Sigma X \vee \Sigma Z \vee \Sigma X \wedge Z$ we get that 
  \[ \Sigma (X \times Z) \simeq \Sigma X \vee \Sigma Z \vee \Sigma X
    \wedge Z \simeq \Sigma Y \vee \Sigma Z \vee \Sigma Y \wedge Z
    \simeq \Sigma (Y \times Z) \]
and thus, as $R \ra A$ is multiplicatively stable
  \[ \cL_X^R(\cL_Z^R(A)) \simeq \cL_{X \times Z}^R(A) \simeq \cL_{Y
      \times Z}^R(A) \simeq \cL_Y^R(\cL_Z^R(A)).\]   
\end{proof}
\begin{remark}
One can interpret Proposition \ref{prop:lodaystable} as the statement
that Loday constructions preserve stability because for all $Z$ there
is an equivalence of augmented commutative $R$-algebras $R \simeq
\cL_Z^R(R)$. 
\end{remark}

The Loday construction behaves nicely with respect to pushouts: 
\begin{lemma} \label{lem:pushouts}
  If $C \leftarrow A \ra B$ is a diagram of cofibrations of
  commutative $R$ algebras and if $A$ is cofibrant as a commutative
  $R$-algebra, then 
  \[ \cL^R_X(C \wedge_A B)  \simeq \cL^R_X(C) \wedge_{\cL^R_X(A)} \cL^R_X(B).\]
\end{lemma}
\begin{proof}
This equivalence is proven using an exchange of priorities in a
colimit diagram based on the equivalence
\[ (C \wedge_A B) \wedge_R (C \wedge_A B) \simeq (C \wedge_R C)
  \wedge_{A \wedge_R A} (B \wedge_R B). \]
\end{proof}
\begin{remark}
  Beware that the above identification does \emph{not} imply that
  multiplicative stability is 
  closed under pushouts in the category of commutative $R$-algebras.  
Knowing that
  $\cL_X^R(D) \simeq \cL_Y^R(D)$ as commutative augmented $D$-algebras
  for $D = A,B$ and $C$
does not imply that
  $\cL^R_X(C \wedge_A B)$ is equivalent to $\cL^R_Y(C \wedge_A B)$
  because we cannot guarantee that the  
  equivalences $\cL_X^R(D) \simeq \cL_Y^R(D)$ commute with the
  structure maps in the pushout diagram.

  For example we know that $H\Q \ra H\Q[t]$ and $H\Q \ra H\Q[t,x]$ are
  multiplicatively stable, but 
  $H\Q \ra H\Q[t]/t^2$ is not stable by \cite{dt}, despite the fact
  that  we can express the latter as a pushout
  $H\Q[t] \wedge_{H\Q[t,x]} H\Q[t]$ where $x$ maps to $t^2$ on the
  left hand side and to $0$ on the right hand side. 
\end{remark}
In the case of $A=R$ we \emph{do} get a stability result:
\begin{corollary} \label{cor:coproducts}
Assume that $R \ra B$ and $R \ra C$ are multiplicatively stable. Then
so is $R \ra B \wedge_R C$.  
\end{corollary}
\begin{proof}
  If $\Sigma X \simeq \Sigma Y$ in $\ssets_*$, then $\cL_X^R(B) \simeq \cL_Y^R(B)$ and
  $\cL_X^R(C) \simeq \cL_Y^R(C)$
  by assumption and these equivalence are of commutative augmented
  $B$- and $C$-algebras, so in particular of  commutative augmented
  $R$-algebras. Note that $\cL_X^R(R) \simeq R$ for all pointed
  $X$. Hence by Lemma \ref{lem:pushouts} we obtain  
  \[ \cL_X^R(B \wedge_R C) \simeq \cL_X^R(B) \wedge_R \cL_X^R(C)
    \simeq \cL_Y^R(B) \wedge_R \cL_Y^R(C) \simeq \cL_Y^R(B \wedge_R C) \]
  and this is an equivalence of commutative augmented $B \wedge_R C$-algebras. 
\end{proof}

\begin{example}
  We know from Proposition \ref{prop:regular} that $HR \ra HR/a$ is
  multipliatively stable for every commutative ring $R$ and every
  regular element $a \in R$. Corollary \ref{cor:coproducts} implies
  that $HR \ra HR/a \wedge_{HR} HR/a$ is multiplicatively stable and
  as before we know that $HR/a \wedge_{HR} HR/a \simeq HR/a \vee
  \Sigma HR/a$, so $HR \ra HR/a \vee \Sigma HR/a$ is multiplicatively
  stable. For instance $H\Z \ra H\Z/p \vee \Sigma H\Z/p$ is
  multiplicatively stable for all primes $p$.  
\end{example}  
\begin{example}
  Taking the coproduct (with a cofibrant model of $H\Z[t]$ as a
  commutative $H\Z$-algebra)  
  \[ \xymatrix{
H\Z \ar[r] \ar[d] & H\Z/p \ar[d] \\
H\Z[t] \ar[r] & H\Z[t] \wedge^L_{H\Z} H\Z/p \simeq H\Z/p[t]} 
\]
shows that $H\Z \ra H\Z/p[t]$ is multiplicatively stable. 
\end{example}
  
\begin{corollary}
Let $R$ be a commutative ring and let $(a_1,\ldots, a_n)$ be a regular
sequence in $R$, then $HR \ra HR/(a_1,\ldots, a_n)$ is multiplicatively stable. 
\end{corollary}
\begin{proof}
 We use induction.  We have shown in Proposition \ref{prop:regular} that $HR\ra HR/a_1$ is multiplicatively stable, so we can inductively assume that $HR\ra HR/(a_1,\ldots, a_{n-1})$ is multiplicatively stable.   We use the fact that the coproduct $HR/(a_1,\ldots, a_{n-1})
  \wedge^L_{HR} HR/a_n$ of
  \[\xymatrix{HR \ar[r] \ar[d] & HR/(a_1,\ldots, a_{n-1})\\ HR/a_n} \]
 is $HR/(a_1, \ldots, a_n)$, and then by Corollary \ref{cor:coproducts} the claim follows.
  
  This identification of the coproduct be proven using the K\"unneth spectral sequence $$\Tor_*^R(R/(a_1,\ldots, a_{n-1}), R/a_n)\Rightarrow \pi_*(HR/(a_1,\ldots, a_{n-1})\wedge^L_{HR} HR/a_n).$$ The Tor can be calculated by tensoring 
  the standard
  free  resolution $\xymatrix{0 \ar[r] & R \ar[r]^{a_n} & R}$ of $R/a_n$ with
  $R/(a_1,\ldots, a_{n-1})$ to  obtain
  \[ \xymatrix@1{0 \ar[r] & R/(a_1,\ldots, a_{n-1}) \otimes_R R
      \ar[rr]^{\mathrm{id}
        \otimes a_n} & & R/(a_1,\ldots, a_{n-1}) \otimes_R R }.\] 
 Since multiplication by $a_n$ is injective on $R/(a_1,\ldots, a_{n-1})$, the $E^2$ term of the spectral sequence consists only of $R/(a_1, \ldots, a_n)$ and we are done.
\end{proof}

\begin{proposition}
  If $S \ra A$ and $S \ra B$ are cofibrations of commutative
  $S$-algebras and if $A$ and $B$ are 
  (multiplicatively) stable, then if $X$ and $Y$ are connected and
  $\Sigma X \simeq \Sigma Y$, then 
  \[ \cL_X^S(A \times B) \simeq \cL_Y^S(A \times B)\]
  as commutative $S$-algebras. 
\end{proposition}
\begin{proof}
  This follows from \cite[Proposition 8.4]{bhlprz} because $\cL_X^S(A
  \times B) \simeq \cL_X^S(A) \times \cL_X^S(B)$ as commutative $S$-algebras. 
\end{proof}

The following notion is investigated in \cite{mccm,mathew}. 
\begin{definition}
  Let $R \ra A \ra B $ be a sequence of cofibrations of commutative
  $S$-algebras with $R$ cofibrant. Then
  this sequence \emph{satisfies \'etale descent} if for all connected
  $X$ the canonical map
  \[ \cL_X^R(A) \wedge_A B \ra \cL_X^R(B)\]
  is an equivalence. 
\end{definition}  
If $R \ra A \ra B$ satisfies \'etale descent and if $X$ is not
connected, so for example 
$X = X_1 \sqcup X_2$ 
with $X_i$ connected for $i=1,2$, then the formula becomes 
\[ \cL^R_X(B) = \cL^R_{X_1 \sqcup X_2}(B) \simeq \cL^R_{X_1}(B)
  \wedge_R \cL_{X_2}^R(B) \simeq 
  \cL^R_{X_1}(A) \wedge_A B \wedge_R \cL_{X_2}^R(A) \wedge_A B.\]

The property of satisfying \'etale descent is closed under smashing
with a fixed commutative $S$-algebra: 
\begin{lemma} \label{lem:smashanded}
If $R \ra A \ra B$ satisfies \'etale descent and if $C$ is a cofibrant
commutative $R$-algebra, then $C \ra C \wedge_R A \ra C \wedge_R B$
satisfies \'etale descent. 
\end{lemma}
\begin{proof}
We know from Lemma \ref{lem:basechange} that $\cL_X^C(C \wedge_R A)
\simeq C \wedge_R \cL_X^R(A)$. Therefore an exchange of pushouts
yields 
\begin{align*}
  \cL_X^C(C \wedge_R A) \wedge_{(C \wedge_R A)} (C \wedge_R B) & \simeq
                                                                 (C
                                                                 \wedge_R
                                                                 \cL^R_X(A))
                                                                 \wedge_{(C
                                                                 \wedge_R
                                                                 A)}
                                                                 (C
                                                                 \wedge_R
                                                                 B)
  \\
                                                               &
                                                                 \simeq
                                                                 (C
                                                                 \wedge_C
                                                                 C)
                                                                 \wedge_{R
                                                                 \wedge_R
                                                                 R}
                                                                 (\cL^R_X(A)
                                                                 \wedge_A
                                                                 B) \\  
  & \simeq C \wedge_R \cL_X^R(B) \simeq \cL_X^{C}(C \wedge_R B).   
\end{align*}  

\end{proof}  

In the case of \'etale descent we can extend stable maps and get maps
that are stable for connected $X$: 
\begin{proposition} \label{prop:etale}
  Let $R \ra A \ra B$ be a sequence of cofibrations of commutative
  $S$-algebras with $R$ cofibrant. If $R \ra A$ 
  is multiplicatively stable and if $R \ra A \ra B$ satisfies \'etale
  descent, then if $\Sigma X \simeq \Sigma 
  Y$ in $\ssets_*$ for connected $X$ and $Y$ we can conclude that there is a weak
  equivalence of augmented commutative 
  $B$-algebras 
  \[ \cL_X^R(B) \simeq \cL_Y^R(B).\]
  \end{proposition}
  \begin{proof}
    As $X$ and $Y$ are connected and as $R \ra A$ is stable, the
    equivalence $\Sigma X \simeq \Sigma Y$ in $\ssets_*$ implies 
    that $\cL_X^R(A) \simeq \cL_Y^R(A)$ and with \'etale descent we
    can upgrade this to 
    \[ \cL_X^R(B) \simeq \cL^R_X(A) \wedge_A B \simeq \cL^R_Y(A)
      \wedge_A B \simeq \cL_Y^R(B).\] 
  \end{proof}

\begin{remark} 
  We know that $H\Q \ra H\Q[t]$ is stable and as $\Q[t]/t^2$ and
  $\Q[t]$ are commutative augmented $\Q$-algebras, we also know that
  $H\Q[t]/t^2 \ra H\Q$ and $H\Q[t] \ra H\Q$ are stable, but since $H\Q
  \ra H\Q[t]/t^2$ and $H\Q \ra H\Q[t]/t^2 \ra H\Q$ are not stable, we
  won't have general descent results. For instance in the diagram
  \[ \xymatrix{
 & \Q[t]/t^2 \ar[d]^\varepsilon \\ 
\Q \ar[ur]^\eta \ar@{=}[r]& \Q
    }\]
  the maps $H(\varepsilon)$ and the identity on $H\Q$ are (even
  multiplicatively) stable, but $H\eta$ isn't. 
  \end{remark}
  
  For morphisms that are faithful Galois extensions and satisfy
  \'etale descent, we obtain a descent result for multiplicative stability: 
\begin{theorem} \label{thm:etaledescent}
  Let $A \ra B$ be a faithful Galois extension with finite Galois
  group $G$ and assume that $A \ra B$ satisfies 
  \'etale descent. Assume that $\Sigma X \simeq \Sigma Y$ for
  connected $X$ and $Y$ implies that there is a $G$-equivariant equivalence $\cL_X^S(B) \simeq 
  \cL_Y^S(B)$ as commutative $B$-algebras. Then also $\cL_X^S(A) \simeq
  \cL_Y^S(A)$ as commutative $A$-algebras.  
\end{theorem}
\begin{proof}
  The base-change result for Galois extensions \cite[Lemma
  7.1.1]{rognes} applied to the diagram 
  \[\xymatrix{
      A \ar[r] \ar[d] & B \ar[d] \\
      \cL_X^S(A) \ar[r] & B \wedge_A \cL_X^S(A)   }\]
  yields that $\cL_X^S(A) \ra B \wedge_A \cL_X^S(A)$ is a $G$-Galois
  extension and by \'etale descent there is an 
  equivalence of augmented commutative $B$-algebras 
$B \wedge_A \cL_X^S(A) \simeq \cL_X^S(B)$ which is $G$-equivariant where
on the left hand side the only non-trivial $G$-action is on the
$B$-factor and on the right hand side $G$-acts on $\cL_X^S(B)$ by
naturality in $B$.
Hence we get a chain of $G$-equivariant equivalences of commutative $B$-algebras
\[ B \wedge_A \cL_X^S(A) \simeq \cL_X^S(B) \simeq \cL_Y^S(B) \simeq B \wedge_A \cL_Y^S(A).\] 
Taking $G$-homotopy fixed points then gives an
equivalence of augmented commutative $A$-algebras 
  \[ \cL_X^S(A) \simeq \cL_X^S(B)^{hG} \simeq \cL_Y^S(B)^{hG} \simeq \cL_Y^S(A).\] 
  
\end{proof}
There exist several definitions of smoothness in the literature (see
for instance \cite{rognes,mccm}) 
using $\THH$-\'etaleness and $\mathsf{TAQ}$-\'etaleness. Using the local behaviour of smooth commutative $k$-algebras
  \cite[Appendix E, Proposition E.2 (d)]{loday} as a template we
  suggest the following 
  variant.  
\begin{definition}
  We call a map of cofibrant $S$-algebras $\varphi \colon R \ra A$
  \emph{really smooth} if it can be 
  factored as $\xymatrix@1{R \ar[r]^(0.4){i_R} & \mathbb{P}_R(X) \ar[r]^(0.6)f &
    A}$ where $i_R$ is the canonical 
  inclusion, $X$ is an $R$-module, and 
  $\xymatrix@1{R \ar[r]^(0.4){i_R} & \mathbb{P}_R(X) \ar[r]^(0.6)f &
    A}$ satisfies \'etale descent.  
\end{definition}

Combining Proposition \ref{prop:etale} and Corollary \ref{cor:pstable} we get:
\begin{proposition} \label{prop:reallysmooth}
If $R \ra A$ is really smooth then $\Sigma X \simeq \Sigma Y$ for
connected $X$ and $Y$ implies
\[ \cL_X^R(A) \simeq \cL_Y^R(A)\]
as commutative $R$-algebras. 
\end{proposition}

The notion of being really smooth is transitive and closed under base change. 
\begin{lemma}
  \begin{itemize}
  \item[]
\item
  If $\varphi \colon R \ra A$ and $\psi \colon A \ra B$ are really
smooth, then so is $\psi \circ \varphi \colon R \ra B$.
\item
If $\varphi \colon R \ra A$ is really smooth and if $C$ is a cofibrant
commutative $R$-algebra, then $C \ra C \wedge_A B$ is really smooth. 
\end{itemize}
\end{lemma}
\begin{proof}
To prove the transitivity,  we take the two given factorizations $\varphi = \xymatrix@1{R
    \ar[r]^(0.4){i_R} & \mathbb{P}_R(X) \ar[r]^(0.6)f & A}$ and $\psi
  = \xymatrix@1{A \ar[r]^(0.4){i_A} & \mathbb{P}_A(Y) \ar[r]^(0.6)g &
    B}$ and combine them to give
  \[ \xymatrix{R \ar[r]^(0.2){i_R} & \mathbb{P}_R(X \vee Y) \simeq
      \mathbb{P}_R(X) \wedge_R \mathbb{P}_R(Y) 
      \ar[rr]^(0.6){f \wedge_R \mathrm{id}} & &  A \wedge_R
      \mathbb{P}_R(Y) \simeq \mathbb{P}_A(Y) \ar[r]^(0.75)g & B  } \]  
  So we have to show that for general maps $f \colon D \ra A, g \colon
  A \ra B, h \colon B \ra C$ of commutative $R$-algebras: 
\begin{enumerate}
\item
  If $f$ satisfies \'etale descent, then so does $f \wedge_R
  \mathrm{id}_C$ for every commutative $R$-algebra $C$. 
\item
  If $g$ and $h$ satisfy \'etale descent, then so does $h \circ g$. 
\end{enumerate}
For (1) let $X$ be connected. As $\cL_X^R(-)$ commutes with pushouts
(see Lemma \ref{lem:pushouts}), we get that $\cL_X^R(A \wedge_R C)
\simeq \cL_X^R(A) \wedge_R \cL^R_X(C)$. As $f$ satisfies \'etale
descent, 
\[ \cL_X^R(A) \wedge_R \cL^R_X(C) \simeq A \wedge_D \cL_X^R(D)
  \wedge_R \cL_X^R(C) \simeq A \wedge_D \cL_X^R(D \wedge_R C) \]
and this in turn is equivalent to $(A \wedge_R C) \wedge_{(D \wedge_R
  C)} \cL_X^R(D \wedge_R C)$. 

The proof of (2) is straightforward because
\begin{align*}
  C \wedge_A \cL_X^R(A) & \simeq C \wedge_B (B \wedge_A \cL_X^R(A)) \\
                        & \simeq C \wedge_B \cL_X^R(B) \\
                        & \simeq \cL_X^R(C).
\end{align*}
For the claim about base change consider the diagram
\[ \xymatrix{
R \ar[r]^{i_R} \ar[d]_\eta & \mathbb{P}_R(X) \ar[d]^f \\ 
C \ar[r] & C \wedge_R \mathbb{P}_R(X).  } \]                      
Adjunction gives that $C \wedge_R \mathbb{P}_R(X) \simeq
\mathbb{P}_C(C \wedge_R X)$. As $R \ra \mathbb{P}_R(X) \ra C$
satisfies \'etale descent we 
obtain with Lemma \ref{lem:smashanded} that $C \ra C \wedge_R
\mathbb{P}_R(X) \ra C \wedge_R A$ satisfies \'etale descent.                   
\end{proof}

\section{Truncated polynomial algebras}
  
Note that we know that the square zero extensions $H\F_p \ra H\F_p
\vee \Sigma^3 H\F_p$ (Example \ref{ex:thh2fp}) and $H\F_p \ra H\F_p
\vee \Sigma H\F_p$ (Corollary \ref{cor:squarezero}) are
multiplicatively stable. However, if place the module $H\F_p$ in 
degree zero, then the following result shows that the square zero extension
$H\F_p \ra H\F_p \vee H\F_p \simeq H\F_p[t]/t^2$ is \emph{not} multiplicatively
stable for odd primes $p$. The proof is a direct adaptation of \cite[\S 3.8]{dt}. 

\begin{theorem} \label{thm:hfpsquarezero}
  Let $p$ be an odd prime. Then $(H\F_p, H\F_p[t]/t^2, H\F_p)$ is not stable. 
\end{theorem}
\begin{corollary} \label{cor:absolute}
The commutative   $H\F_p$-algebra $H\F_p[t]/t^ 2$ is neither 
multiplicatively stable nor $H\F_p[t]/t^ 2$-linearly stable over $H\F_p$. 
\end{corollary}
\begin{proof}
  If it were, then this would imply that $(H\F_p, H\F_p[t]/t^2, H\F_p)$ is 
  stable. 
\end{proof}
\begin{remark}
  In \cite[Theorem 4.18]{hklrz} we extend Theorem \ref{thm:hfpsquarezero} to $\F_p[t]/t^n$ for
  $2 \leq n < p$. 
\end{remark}
\begin{proof}[Proof of Theorem \ref{thm:hfpsquarezero}]
We know that
\[ \pi_*\cL_{S^1 \vee S^1 \vee S^2}^{H\F_p}(H\F_p[t]/t^2; H\F_p) \cong
    \pi_*\cL_{S^1}^{H\F_p}(H\F_p[t]/t^2; H\F_p)^{\otimes_{\F_p} 2}
\otimes_{\F_p}
    \pi_*\cL_{S^2}^{H\F_p}(H\F_p[t]/t^2; H\F_p)\] 
and by \cite{bhlprz} we know what the tensor factors are: 
  \[\pi_*\cL_{S^1}^{H\F_p}(H\F_p[t]/t^2; H\F_p) \cong
    \HH_*^{\F_p}(\F_p[t]/t^2; \F_p) \cong 
    \Lambda_{\F_p}(\varepsilon t) \otimes_{\F_p} \Gamma_{\F_p}(\varphi^0 t)\] 
and
  \[ \pi_*\cL_{S^2}^{H\F_p}(H\F_p[t]/t^2; H\F_p) \cong
    \HH_*^{[2],\F_p}(\F_p[t]/t^2; \F_p) \cong \Gamma_{\F_p}(\varrho^0
    \varepsilon t) \otimes \bigotimes_{k} (\Lambda_{\F_p}(\varepsilon
    \varphi^k t) \otimes \Gamma_{\F_p}(\varphi^0\varphi^k t)).\] 
Torus homology is the total complex of the bicomplex for 
$\cL_{S^1 \times S^1}^{\F_p}(\F_p[t]/t^2; \F_p)$ as in \cite{dt}.
In the bicomplex in bidegree $(n,m)$ we have  the term
  \[ \cL_{[n] \times [m]}^{\F_p}(\F_p[t]/t^2; \F_p) \cong \F_p \otimes_{\F_p}
    (\F_p[t]/t^2)^{(n+1)(m+1)-1} \cong (\F_p[t]/t^2)^{(n+1)(m+1)-1}.\] 
In total degree one we have contributions from $(0,1)$ and $(1,0)$
that we call $y_1^v$ and $y_1^h$ 
as in \cite[\S 3.8]{dt}. Everything is a cycle here and these elements
correspond to $\begin{matrix} 1 \\ 
  \otimes \\
  t\end{matrix}$ and $1 \otimes t$. 

From now on we suppress the tensor signs from the notation and we
denote the generators by matrices.  
In total degree two there are three possibilities $(0,2)$, $(1,1)$ and
$(2,0)$. There are the classes $y_2^v$ in bidegree $(0,2)$, and $y_2^h$
in bidegree $(2,0)$ corresponding 
to the standard Hochschild generators $\begin{pmatrix} 1 \\ t \\
  t\end{pmatrix}$ and 
$\begin{pmatrix} 1 &  t & t \end{pmatrix}$. 

In bidegree $(1,1)$ there are the following possibilities for
non-degenerate cycles:  
\[ \begin{pmatrix} 1 & t \\ t & t \end{pmatrix}, 
\begin{pmatrix} 1 & t \\ t & 1 \end{pmatrix}, 
\begin{pmatrix} 1 & t \\ 1 & t \end{pmatrix}, 
\begin{pmatrix} 1 & 1 \\ t & t \end{pmatrix}, \text{ and } 
\begin{pmatrix} 1 & 1 \\ 1 & t \end{pmatrix}.  \] 
As we are working over $\F_p$ for an odd prime $p$, $2$ is invertible. 
The boundary of $\frac{1}{2}\begin{pmatrix} 1 & 1 \\ 1 & t \\ 1 &
  t\end{pmatrix}$ is  
$\begin{pmatrix} 1 & t \\ 1 & t \end{pmatrix}$, the boundary of 
$\frac{1}{2}\begin{pmatrix} 1 & 1 & 1 \\ 1 & t & t\end{pmatrix}$ is 
$\begin{pmatrix} 1 & 1 \\ t & t \end{pmatrix}$. 
Finally, we identify $\begin{pmatrix} 1 & t \\ t & t \end{pmatrix}$ as
the boundary of $\begin{pmatrix} 1 & 1 \\1 & t \\ t &
  t \end{pmatrix}$.   

\noindent
The element $\begin{pmatrix} 1 & 1 \\ t & 1 \\ 1 & t \end{pmatrix}$
ensures that  
$\begin{pmatrix} 1 & t \\ t & 1 \end{pmatrix}$ is homologous to 
$\begin{pmatrix} 1 & 1 \\ t & t \end{pmatrix}$, so we are left with
the generator in bidegree  $(1,1)$ given by $\begin{pmatrix} 1 & 1 \\ 1
  & t \end{pmatrix}$.  

So we get (at most) a $3$-dimensional vector space in total degree $2$. 

\medskip
In $\pi_2\cL^{H\F_p}_{S^1 \vee S^1 \vee S^2}(H\F_p[t]/t^2; H\F_p)$, 
however, we get the generators 
  $\varphi^0 t \otimes 1 \otimes 1$, $1 \otimes \varphi^0 t \otimes
  1$,  $\varepsilon t \otimes 
  \varepsilon t \otimes 1$ and $1 \otimes 1 \otimes \varrho^0
  \varepsilon t$, so we have a $4$-dimensional vector space. 
\end{proof}
\begin{remark}
  For odd primes $2$ is invertible and this reduces the number of
  generators in total degree $2$ to $3$ in the torus homology of
  $\F_p[t]/t^2$ over $\F_p$ with $\F_p$-coefficients. For $p=2$ one
  can check that there is an extra class coming from $\begin{pmatrix}
    1 & t \\ t & 1 \end{pmatrix}$  which is homologous to
  $\begin{pmatrix} 1 & 1 \\ t & t \end{pmatrix}$ and to
  $\begin{pmatrix} 1 & t \\ 1 & t \end{pmatrix}$ so together with the
  class $\begin{pmatrix} 1 & 1 \\ 1 & t \end{pmatrix}$ this gives two
  generators in bidegree $(1,1)$ and the ones in $(2,0)$ and $(0,2)$
  giving a total of dimension $4$. As $\F_2[t]/t^2$ is a commutative
  Hopf algebra over $\F_2$, we know that $\F_2 \ra \F_2[t]/t^2$ and
  $\F_2  \ra \F_2[t]/t^2 \ra \F_2$ are stable. 
\end{remark}

We can model $S[t]/t^n$ as $S \wedge \Pi_+$ for the commutative pointed
monoid $\Pi_+=\{+,1,t,\ldots, t^{n-1}\}$. In \cite[Theorem 4.18]{hklrz} we generalize the result from 
Theorem \ref{thm:hfpsquarezero} to all $2 \leq n < p$. 
\begin{corollary} \label{cor:sttsquare}
For every $n \geq 2$ the map $S \ra S[t]/t^n$ is not multiplicatively stable . 
\end{corollary}
\begin{proof}
If it were multiplicatively stable, then by Lemma \ref{lem:basechange} $H\F_p \ra
H\F_p[t]/t^n$ would be as well.  For $n=2$ this contradicts the result above. For higher $n$,
there is a prime $p$ with $p > n$, and then \cite[Theorem 4.18]{hklrz} yields that $H\F_p \ra
H\F_p[t]/t^n$ isn't multiplicatively stable. 
\end{proof}

\begin{remark}
  Neither stability nor multiplicative stability are transitive:
for every commutative ring $k$ the map $k \ra k[t]$ is smooth, hence
(multiplicatively) stable and $k[t] \ra k[t]/t^2$ is stable by Proposition
\ref{prop:regular}, but for $k = \Q$ Dundas and Tenti show \cite{dt}
that $\Q \ra \Q[t]/t^2$ is not stable and for $k=\F_p$ the know that $\F_p \ra
\F_p[t]/t^2$ is not multiplicatively stable.   
\end{remark}

\begin{proposition} \label{prop:pmonoids}
  Let $k$ be a field and let $\Pi_+$ be a pointed commutative
  monoid. If $S \ra Hk$ and $Hk \ra Hk[\Pi_+]$  are multiplicatively
  stable, then $\Sigma X \simeq \Sigma Y$ in $\ssets_*$ implies that
  $\cL_X^S(Hk[\Pi_+]) \simeq \cL_Y^S(Hk[\Pi_+])$ as augmented commutative
  $Hk$-algebras.    
\end{proposition}
\begin{proof}
This follows from the splitting of $\cL_X^S(Hk[\Pi_+])$ as a
commutative augmented $Hk$-algebra  \cite[Theorem 7.1]{hm} as  
\begin{equation} \label{eq:splitting} 
  \cL_X^S(Hk[\Pi_+]) \simeq \cL_X^S(Hk) \wedge_{Hk} \cL_X^{Hk}(Hk[\Pi_+]).
  \end{equation}
\end{proof}

It is important to know whether $S \ra Hk$ is
$Hk$-linearly stable, because if it is, then for all $Hk$-linearly stable $Hk \ra HA$ that satisfy a splitting
formula as in \eqref{eq:splitting}, such as polynomial algebras, 
we would get that $S \ra HA$ is $Hk$-linearly stable.

Of course, $S \ra H\Q$ is multiplicatively stable because $S_\Q \simeq H\Q$ and
$H\Q \wedge_S H\Q \simeq H\Q$. 
We do not know whether $S \ra H\F_p$ is stable. We can express $H\F_p$ as a Thom spectrum, but this
Thom spectrum structure comes from a double loop map, so it is not of the form needed for Corollary
\ref{cor:thomstable}. So we leave this as an open question:

\begin{center}
Is $H\F_p$ stable? 
 \end{center} 

\bigskip 

We close with a family of examples that show that several of the Juggling Formulas from \cite{bhlprz} cannot
be generalized to arbitrary pointed simplicial sets because that would contradict certain non-stability results. 

Let $k$ be a field. The case $k \ra k[t] = R \ra k[t]/t^m = R/t^m \ra
k = R/t$ for $m \geq 2$ is special in the 
sense that the quotient $k[t]/t^m$ is itself a commutative augmented
$k$-algebra, so we can combine our splitting result for higher order
Shukla homology \cite[Proposition 7.5]{bhlprz} with 
the Juggling Formula \cite[Theorem 3.3]{bhlprz}. We have 
 \cite[Theorem 7.6]{bhlprz}: 
\begin{equation} \label{eq:tatesplitting} 
\THH^{[n]}(k[t]/t^m; k) \simeq \THH^{[n]}(k[t]; k) \wedge_{Hk}
\THH^{[n], Hk[t]}(k[t]/t^m; k) 
\end{equation}
for all $n \geq 1$ and for all $m \geq 2$. In this special case we
can get the following $Hk$-version of this result: 

\begin{theorem}
  Let $k$ be a field and let $m$ be greater or equal to $2$. Then for
  all $n \geq 1$ 
\begin{equation} \label{eq:augmsplitting}
  \HH^{[n],k}(k[t]/t^m; k) \simeq \HH^{[n],k}(k[t]; k) \wedge^L_{Hk}
  \THH^{[n],Hk[t]}(k[t]/t^m; k).
  \end{equation}
\end{theorem}

\begin{proof}
  Consider the diagram
  \[\xymatrix{
{Hk} \ar[r] \ar[d] & {\THH^{[n-1],Hk[t]}(k)} \ar[r]\ar[d] &
{\HH^{[n],k}(k[t];k)} \ar[d]\\ 
{\THH^{[n],Hk[t]}(k[t]/t^m; k)} \ar[r] & {\THH^{[n-1],Hk[t]/t^m}(k)}
\ar[r] & {\HH^{[n],k}(k[t]/t^m; k).} 
    } \]

  The left-hand square is a homotopy pushout square by
  \cite[Proposition 7.5]{bhlprz} and the juggling formula
  \cite[Theorem 3.3]{bhlprz} applied to $Hk \ra Hk[t] \ra Hk[t]/t^m
  \ra Hk$ ensures that the right-hand square is also a homotopy
  pushout square because for all $n \geq 1$
  \[ \HH^{[n],k}(k[t]/t^m; k) \simeq \HH^{[n],k}(k[t];k)
    \wedge^L_{\THH^{[n-1],Hk[t]}(k)}  \THH^{[n-1],Hk[t]/t^m}(k). \] 
  This yields that the outer rectangle is also a homotopy pushout
  square and this was the claim. 
  \end{proof}
  
  \begin{remark}
    Note that there cannot be a version of \eqref{eq:tatesplitting}
    and \eqref{eq:augmsplitting} for arbitrary connected $X$: We know
    that $Hk[t] \ra Hk[t]/t^2 \ra Hk$ is multiplicatively stable for
    all fields $k$ and we know that $Hk \ra Hk[t] \ra Hk$ is
    stable. But for any odd prime $p$ we know that $H\F_p \ra
    H\F_p[t]/t^2 \ra H\F_p$ is not stable and that there is an actual
    discrepancy between
    \[\pi_2(\cL_{S^1 \times
      S^1}^{H\F_p}(H\F_p[t]/t^2; H\F_p)) \not\cong \pi_2(\cL_{S^1 \vee S^1
      \vee S^2}^{H\F_p}(H\F_p[t]/t^2; H\F_p)), \]
  so there cannot be an
    equivalence between $\cL_{S^1 \times S^1}^{H\F_p}(H\F_p[t]/t^2;
    H\F_p)$ and $\cL_{S^1 \vee S^1 \vee S^2}^{H\F_p}(H\F_p[t]/t^2;
    H\F_p)$.  
\end{remark}

\section{Thom spectra and topological K-theory}

Christian Schlichtkrull gives a closed formula for the Loday construction
on Thom spectra \cite[Theorem 1.1]{schlichtkrull}: Let $f \colon W \ra
BF_{hI}$ be an $E_\infty$-map 
with $W$ grouplike and let $T(f)$ denote the corresponding Thom
spectrum. Then for any $T(f)$-module spectrum $M$ one has
\begin{equation} \label{eq:thom}
  \cL_X^S(T(f); M) \simeq M \wedge \Omega^\infty(E_W \wedge X)_+
\end{equation}
where $E_W$ is the Omega spectrum associated to $W$ (\ie, $W \simeq
\Omega^\infty E_W$). If $M$ is a 
commutative $T(f)$-algebra spectrum, then 
the above equivalence is one of commutative $T(f)$-algebras. For $M =
T(f)$ the equivalence also respects the augmentation. 

An immediate consequence of Schlichtkrull's result is the following: 
\begin{corollary} \label{cor:thomstable}
If $T(f)$ is a Thom spectrum as above, then $S \ra T(f)$ is
multiplicatively stable.  
\end{corollary}
\begin{proof}
  If $\Sigma X \simeq \Sigma Y$ in $\ssets_*$, then on the level of spectra we obtain
  \[ \Sigma(E_W \wedge X) \simeq E_W \wedge \Sigma X \simeq E_W \wedge
    \Sigma Y \simeq \Sigma(E_W \wedge Y),\]
  but here suspension is invertible, hence
  $E_W \wedge X \simeq E_W \wedge Y$ and therefore
  \[ \cL_X^S(T(f)) \simeq \cL_Y^S(T(f)).\]

An equivalence of spectra induces an equivalence of infinite
loop spaces and the $T(f)$-algebra structure on $T(f) \wedge
\Omega^\infty(E_W \wedge X)_+$ just 
comes from the one on $T(f)$ and the infinite loop structure on
$\Omega^\infty(E_W \wedge X)$. This gives the multiplicativity of the
stability.  
\end{proof}

The case of the suspension spectrum of an abelian topological group is
a special case where we take $f \colon G \ra BF_{hI}$ to be the
trivial map. Then $T(f) \simeq \Sigma^\infty_+(G)$. Other examples are
$MU$, $MO$, $MSO$, $MSp$ or $MSpin$.

\begin{remark}
  Nima Rasekh, Bruno Stonek, and Gabriel Valenzuela \cite[Theorem 4.11]{rsv} generalize Schlichtkrull's calculation
  to generalized Thom spectra, \ie, Thom spectra that are formed with respect to a map of $E_\infty$-groups $f \colon
  G \ra \mathrm{Pic}(R)$ for some commutative ring spectrum $R$. They note (see \cite[Remark 4.14]{rsv}) that
  this implies stability for such Thom spectra. 
\end{remark}

\begin{remark}
  Note that by Corollary \ref{cor:sttsquare} spherical abelian monoid
  rings are \emph{not} stable in general, whereas spherical abelian group rings are. 

\end{remark}

Bruno Stonek calculates higher $\THH$ of periodic complex topological
K-theory, $KU$, and he determines topological Andr\'e-Quillen 
homology of $KU$ \cite{stonek}. He uses Snaith's description of $KU$
as the Bott localization of $\Sigma^\infty_+ \mathbb{C}P^\infty$. The
latter is a Thom spectrum because $\mathbb{C}P^\infty = BU(1)$ can be
realized as a topological abelian group.

\begin{theorem} \label{thm:KU}
  If $X$ and $Y$ are connected and $\Sigma X \simeq \Sigma Y$ in $\ssets_*$, then
  \[ \cL_X^S(KU) \simeq \cL_Y^S(KU)\]
  as commutative augmented $KU$-algebra spectra. 
\end{theorem}
\begin{proof}
Let $\beta$ denote the Bott element. Stonek uses Snaith's identification of $KU$ as
  $\Sigma^\infty_+\mathbb{C}P^\infty[\beta^{-1}]$ to prove
  \cite[Corollary 4.12]{stonek} that there is a zigzag of equivalences 
  \[ \xymatrix{
{\THH(KU) \simeq  \THH(\Sigma^\infty_+\mathbb{C}P^\infty[\beta^{-1}])} &
\ar[l]_(0.5)\simeq \ar[d]^\simeq 
\THH(\Sigma^\infty_+\mathbb{C}P^\infty)
\wedge_{\Sigma^\infty_+\mathbb{C}P^\infty} 
      \Sigma^\infty_+\mathbb{C}P^\infty[\beta^{-1}]  \\
&       (\THH(\Sigma^\infty_+\mathbb{C}P^\infty))[\beta^{-1}]. 
    } \]
The same argument yields that for any connected $X$ the localization
of $\cL_X^S(\Sigma^\infty_+\mathbb{C}P^\infty)$ at $\beta$ is equivalent 
  to $\cL_X^S(\Sigma^\infty_+\mathbb{C}P^\infty[\beta^{-1}]) = \cL_X^S(KU)$. 

The localization map $\Sigma^\infty_+\mathbb{C}P^\infty
\ra \Sigma^\infty_+\mathbb{C}P^\infty[\beta^{-1}]$ satisfies \'etale
descent, and therefore the composite
$S \ra \Sigma^\infty_+\mathbb{C}P^\infty \ra
\Sigma^\infty_+\mathbb{C}P^\infty[\beta^{-1}]$ identifies 
$KU$ as an \'etale extension of a Thom spectrum. By Proposition
\ref{prop:etale} we obtain multiplicative stability for connected
simplicial sets.  
\end{proof}

\begin{corollary}
  If $X$ and $Y$ are connected simplicial sets with $\Sigma X \simeq
  \Sigma Y$ then $\cL_X^S(KO) \simeq \cL_Y^S(KO)$ as commutative
  $KO$-algebras.  
\end{corollary}
\begin{proof}
  Rognes shows \cite[\S 5.3]{rognes} that the complexification map $KO
  \ra KU$ is a faithful $C_2$-Galois extension of commutative ring
  spectra and Mathew \cite[Example 4.6]{mathew} deduces from
  \cite[Example 5.9]{cmnn} that it satisfies \'etale descent. Schlichtkrull's equivalence
  from \eqref{eq:thom} is natural hence it preserves the $C_2$-action. Therefore 
  the result follows from Theorem \ref{thm:etaledescent}. 
\end{proof}

In \cite{chy} Hood Chatham, Jeremy Hahn, and Allen Yuan construct interesting examples of
$E_\infty$-ring spectra. For a prime $p$ they consider the infinite loop space
\[ W_h = \Omega^\infty \Sigma^{2\nu(h)}BP\langle h\rangle \]
where $\nu(h) = \frac{p^{h+1}-1}{p-1}$. This is the $2\nu(h)$th space of the Omega spectrum
for the $h$-truncated Brown-Peterson spectrum $BP\langle h\rangle$; these spaces were extensively studied
by Steve Wilson \cite{wilson}. On the suspension spectrum
of $W_h$ they invert the generator
$x$ of the bottom non-trivial homotopy group $\pi_{2\nu(h)}(W_h) \cong \Z_{(p)}$ and obtain an
$E_\infty$-ring spectrum
\[ R_h := (\Sigma^\infty_+ W_h)[x^{-1}]\]
which has remarkable features \cite[Theorem 1.13]{chy}: $R_h$ has torsion-free homotopy groups that
vanish in odd degrees, it is Landweber exact, and its Morava-$K(n)$ localization
$L_{K(n)}R_h$ vanishes if and only if $n > h+1$,
so $R_h$ is of chromatic height $n+1$. As $W_0$ is $\bC P^\infty$, this recovers Snaith's construction
in this special case, but there are many more interesting examples. For all of these spectra, the above method
of proof applies, so we obtain. 
\begin{theorem}
If $X$ and $Y$ are connected and $\Sigma X \simeq \Sigma Y$ in $\ssets_*$, then
  \[ \cL_X^S(R_h) \simeq \cL_Y^S(R_h)\]
  as commutative augmented $R_h$-algebra spectra.
\end{theorem}

\section{The Greenlees spectral sequence}
Let $k$ be a field and let $A \ra B$ be a morphism of connective
commutative $S$-algebras with an 
augmentation to $Hk$ satisfying some mild finiteness assumption. Then
by \cite[Lemma 3.1]{greenlees} there is a spectral sequence
\[E^2_{s,t} = \pi_s(B \wedge_A Hk) \otimes_k \pi_t(A) \Rightarrow \pi_{s+t}(B).\]

Let $p$ be an odd prime. 
We consider the cofibration $S^1 \vee S^1 \hookrightarrow S^1 \times
S^1 \ra S^2$ and the associated pushout diagram 
\[ \xymatrix{ \cL_{S^1 \vee S^1}^R(H\F_p[t]/t^2; H\F_p) \ar[r] \ar[d] &
    \cL_{S^1 \times S^1}^R(H\F_p[t]/t^2; H\F_p) 
   \ar[d] \\
H\F_p \ar[r] & \cL_{S^2}^R(H\F_p[t]/t^2; H\F_p).  } \]
Here, $R$ can be $S$ or $H\F_p$. For $R=S$ we obtain a Greenlees
spectral sequence 
\begin{equation} \label{eq:grS}
  \pi_s(\cL_{S^2}^S(H\F_p[t]/t^2; H\F_p)) \otimes_{\F_p} \pi_t(\cL_{S^1
    \vee S^1}^S(H\F_p[t]/t^2; H\F_p)) \Rightarrow 
  \pi_{s+t}(\cL_{S^1 \times S^1}^S(H\F_p[t]/t^2; H\F_p))
\end{equation}
whereas for $R = H\F_p$ the spectral sequence is
\begin{equation} \label{eq:grfp}
  \pi_s(\cL^{H\F_p}_{S^2}(H\F_p[t]/t^2; H\F_p)) \otimes_{\F_p}
  \pi_t(\cL^{H\F_p}_{S^1 \vee S^1}(H\F_p[t]/t^2; H\F_p)) 
  \Rightarrow 
  \pi_{s+t}(\cL^{H\F_p}_{S^1 \times S^1}(H\F_p[t]/t^2; H\F_p)).
\end{equation}
In \eqref{eq:grS} every term $\cL_X^S(H\F_p[t]/t^2; H\F_p)$ splits as
\[ \cL_X^S(H\F_p) \wedge_{H\F_p} \cL_X^{H\F_p}(H\F_p[t]/t^2; H\F_p) \]
naturally in $X$, and going from working over $S$ to working over
$H\F_p$ simply collapses the $ \cL_X^S(H\F_p)$ to $ \cL_X^{H\F_p}
(H\F_p)\simeq H\F_p$.  Therefore we get a surjection of the spectral 
sequence of \eqref{eq:grS} onto the one of 
 \eqref{eq:grfp}, and if all the spectral sequence differentials
 vanish on the former, they have to vanish on the latter too.  But we
 know that the rank of $\pi_2(\cL^{H\F_p}_{S^1 \times S^1}(H\F_p[t]/t^2;
 H\F_p))$ is less  than the rank of the $E^2$-term in total degree $2$,
 hence there has to be a non-trivial differential 
in \eqref{eq:grfp} and hence also in \eqref{eq:grS}. This implies the
following result.  
\begin{theorem} \label{thm:notstable}
For every odd prime $p$, $(S, H\F_p[t]/t^2, H\F_p)$ is not stable. 
\end{theorem}
With the results of \cite[\S 4]{hklrz} the above result can be generalized to 
$\F_p[t]/t^n$ for $2 \leq n < p$. 

\medskip
Instead of stability we can consider the following property of Loday
constructions. 
\begin{definition}
  Let $R$ be a cofibrant commutative ring spectrum and let 
  $R \ra A \ra C$ be a sequence of cofibrations of commutative
  $R$-algebras. We say that 
  $R \ra A \ra C$ \emph{decomposes products} if for all pointed connected 
  simplicial sets $X$ and $Y$ 
\[ \cL^R_{X \times Y}(A; C) \simeq \cL^R_{X \vee Y \vee X \wedge Y}(A; C)\]  
\end{definition}
\noindent 
Note that the right hand side is equivalent to $\cL^R_X(A; C) \wedge_C
\cL^R_Y(A; C) \wedge_C \cL_{X \wedge Y}^R(A; C)$. 

\begin{proposition}
Let $R \ra A \ra B \ra A \ra Hk$  be a sequence of commutative
$S$-algebras that turns $B$ into an augmented commutative
$A$-algebra. Assume that $k$ is a field. 

If $R \ra B \ra Hk$ decomposess products then so does $R \ra A \ra Hk$. 
\end{proposition}
\begin{proof}
  The naturality of the Loday construction ensures that the vertical
  compositions in the diagram 
  \[\xymatrix{
      \cL^R_{X \vee Y}(A; Hk) \ar[r]\ar[d]& \cL^R_{X \times Y}(A; Hk)
      \ar[r] \ar[d]& \cL^R_{X \wedge Y}(A; Hk) \ar[d] \\ 
      \cL^R_{X \vee Y}(B; Hk)\ar[r]\ar[d]& \cL^R_{X \times Y}(B; Hk)
      \ar[r] \ar[d]& \cL^R_{X \wedge Y}(B; Hk) \ar[d] \\ 
      \cL^R_{X \vee Y}(A; Hk)\ar[r]& \cL^R_{X \times Y}(A; Hk) \ar[r]
      & \cL^R_{X \wedge Y}(A; Hk)  
    } \]
  are the identity. Therefore the spectral sequence
  \[  \pi_s(\cL^R_{X \wedge Y}(A; Hk)) \otimes_k \pi_t(\cL^R_{X \vee
      Y}(A; Hk)) \Rightarrow \pi_{s+t}(\cL^R_{X \times Y}(A; Hk))  \]
  is a direct summand of the one for $B$. So if the spectral sequence
  for $A$ had a non-trivial differential, then the one for $B$ also
  had to have one, but as $B$ decomposess products, this cannot
  happen.  
\end{proof}
Note that this gives an additive splitting, but we can't rule out
multiplicative extensions.

If $B$ does not decompose products, then this does \emph{not} imply
that $A$ doesn't either. A concrete counterexample is $S \ra H\Q \ra
H\Q[t]/t^2 \ra H\Q$. Here, $S \ra H\Q[t]/t^2$ does not decompose
products, but $S \ra H\Q$ is even multiplicatively stable.

\section{Rational Equivalence} \label{sec:rational}
The starting point for this section is the following result:
\begin{proposition} \label{stableequiv}  (Berest, Ramadoss, and Yeung \cite{bry}) If $k$ is a field of
  characteristic zero and $A$ is a commutative Hopf algebra over $k$, then any rational equivalence
  $f \colon X \to Y$ between simply connected spaces induces a weak equivalence 
\begin{equation*}
f_* \colon \cL^{Hk}_X({HA}) \simeq \cL^{Hk}_Y({HA}).
\end{equation*}

If  $f \colon X \to Y$ is a rational equivalence between simply connected pointed spaces then $f$
induces a weak equivalence 
\begin{equation*}
f_* \colon \cL^{Hk}_X({HA; Hk}) \simeq \cL^{Hk}_Y({HA;Hk}).
\end{equation*}

\end{proposition}
\begin{proof}
  This follows from \cite[Theorem 1.3 (a)]{bry} which says that for such $k$ and $A$ and any unbased
  simplicial set $X$, 
\[\pi_* \cL^{Hk}_X({HA}) \cong \HR_*(\Sigma(X_+), A)\]
where $\HR$ is representation homology, and from  \cite[Proposition 4.2]{bry}, which says that
rational equivalences 
between simply connected spaces induce isomorphisms on representation homology for such $k$ and $A$.
In the pointed setting, \cite[Theorem 1.3 (b)]{bry} applies to give the equivalence
\[ \pi_*\cL^{Hk}_X({HA; Hk}) \cong \HR_*(\Sigma X, A; k)\]
\end{proof}



We can extend 
Proposition \ref{stableequiv} 
to augmented commutative 
finitely generated $k$-algebras:

\begin{proposition}
  If $k$ is a field of characteristic zero and $A$ is a finitely generated augmented commutative
  $k$-algebra,
  then  any rational equivalence $f \colon  X\to Y$ of simply connected spaces induces a weak
  equivalence 
\begin{equation*}
f_* \colon  \cL^{Hk}_X({HA};Hk) \simeq \cL^{Hk}_Y({HA};Hk).
\end{equation*}
\end{proposition}

\begin{proof}
Let $A$ be generated by $a_1, a_2,\ldots, a_\ell$ as a commutative $k$-algebra, let
  $\varepsilon \colon  A\to k$ be its augmentation, and let $\eta \colon k\to A$ be the unit map. 
We denote by $I$ the augmentation ideal, $I=\ker\varepsilon$. Then for all $1\leq i\leq\ell$,
$a_i-\eta(\varepsilon(a_i))\in I$, so we can define a surjection of augmented commutative $k$-algebras 
\[\varphi \colon k[x_1,\ldots, x_\ell] \to A,
\quad \varphi(x_i)=a_i-\eta(\varepsilon(a_i)) \text{ for all }   1\leq i\leq\ell. 
\]
Here  we consider the augmentation of $ k[x_1,\ldots, x_\ell]$ that sends every $x_i$ to zero, so that
its augmentation ideal is $(x_1, \ldots, x_\ell)$. 

Since $k$ is a field,  $k[x_1,\ldots, x_\ell]$ is Noetherian so we can find finitely many polynomials
$f_1, \ldots, f_m$ in the $x_i$ to generate $\ker \varphi$.  Since $\varepsilon \circ \varphi$ is the
augmentation of $k[x_1,\ldots, x_\ell]$,  we get that for all $1\leq j\leq m$,  $f_j$ is an element
in $(x_1,\ldots, x_\ell)$.
Hence, we can define another map of augmented commutative $k$-algebras
\[\psi \colon   k[u_1,\ldots, u_m] \to k[x_1,\ldots, x_\ell], 
\quad \psi(u_i)=f_j \text{ for  all }   1\leq j \leq m,
\]
which maps $k[u_1,\ldots, u_m] $ onto $\ker \varphi$.  The augmentation of $k[u_1,\ldots, u_m]$  is
again the standard one. We express $A$ as a  pushout of commutative augmented $k$-algebras
\[ \xymatrix{  k[u_1,\ldots, u_m]  \ar[r]^\psi  \ar[d] &
  k[x_1,\ldots, x_\ell]   
   \ar[d] ^\varphi \\
k \ar[r] ^\eta & A  } \]
where all entries except $A$ are known to be commutative Hopf algebras over $k$. So for them, $f$
induces a weak equivalence 
$\cL^{Hk}_X (H(-); Hk) \to \cL^{Hk}_Y (H(-); Hk)$.
Since both  $\cL^{Hk}_X (H(-); Hk) $ and $\cL^{Hk}_Y (H(-); Hk)$ send pushouts of augmented commutative $k$-algebras to
homotopy pushouts of augmented commutative $Hk$-algebras, $f$ also induces a weak equivalence on the pushout.
\end{proof}



Let $X$ be a pointed simply connected simplicial set. Then rationally
\[ \Sigma X_\Q \simeq \bigvee_{i \in I} S^{k_i}_\Q \]
for some indexing set $I$ and some $k_i \geq 2$ (see for instance \cite[Theorem 24.5]{fht}). In
particular, 
\[ \bigvee_{i \in I} S^{k_i}_\Q \simeq \Sigma\left(\bigvee_{i \in I} S^{k_i-1}_\Q\right). \]
So with the help of \cite[Theorem 1.3]{bry} we obtain:
\begin{theorem}
  For every pointed simply connected $X$, every field of characteristic zero $k$ and every commutative
  Hopf-algebra $A$ over $k$, for a suitable indexing set $I$ and integers $k_i \geq 2$ we get
  \[ \pi_*\cL^k_{X_\Q}(A; k) \cong \pi_*\cL^k_{\bigvee_{i \in I} S^{k_i-1}_\Q}(A; k). \]
\end{theorem}  

For simply-connected spaces and $k$ and $A$ as above we know by \cite[Proposition 4.2]{bry} that
the homotopy type of the Loday construction only depends on the rational homotopy type of the suspension,
so we can
discard the rationalization in the above statement. This yields, for instance:
\begin{example}
  Let $X = \mathbb{C}P^n$ for some $n \geq 1$.  Then for every field of characteristic zero $k$ and every
  commutative Hopf-algebra $A$ over $k$., as $\Sigma\bC P^n_\Q \simeq \Sigma \bigvee_{i=1}^n S^{2i}_\Q$, we
  obtain 
  \[ \pi_*\cL^k_{\bC P^n}(A; k) \cong \pi_*\cL^k_{\bigvee_{i=1}^n S^{2i}}(A; k). \]
\end{example}

\end{document}